\documentclass[amssymb, 11pt]{amsart}
\usepackage{amsmath}
\usepackage{amsthm}
\usepackage{amsfonts}
\usepackage{ytableau}

\newtheorem{thm}{Theorem}[section]
\newtheorem{theorem}[thm]{Theorem}

\newtheorem{cor}[thm]{Corollary}

\newtheorem{lemma}[thm]{Lemma}

\usepackage{varioref}
\labelformat{section}{Section~#1}
\labelformat{figure}{Figure~#1}
\labelformat{table}{Table~#1}

\usepackage[numbers]{natbib}
\date{}

\begin{document}

\title [Cohen Lenstra Partitions] {Cohen Lenstra Partitions and Mutually Annihilating Matrices over a Finite Field}

\author{Jason Fulman}
\address{Department of Mathematics\\
        University of Southern California\\
        Los Angeles, CA, 90089, USA}
\email{fulman@usc.edu}

\author{Robert Guralnick}
\address{Department of Mathematics\\
        University of Southern California\\
        Los Angeles, CA, 90089, USA}
\email{guralnic@usc.edu}


\date{November 14, 2021}

%

\begin{abstract}
Motivated by questions in algebraic geometry, Yifeng Huang recently derived generating functions for counting mutually annihilating matrices and mutually annihilating nilpotent matrices over a finite field. We give a different derivation of his results using statistical properties of random partitions chosen from the Cohen-Lenstra measure.
\end{abstract}

\maketitle

\section{Introduction}

Motivated by questions in algebraic geometry, Yifeng Huang \cite{Hu} derived the following two generating functions.

\begin{equation} \label{eq1}
 \sum_{n \geq 0} \frac{|\{A,B \in Mat_n(F_q):AB=BA=0\}|}{|GL(n,q)|} u^n
= \frac{1}{1-u} \sum_{a \geq 0} \frac{u^a}{(1/q)_a (u/q)_a}.
\end{equation}

\begin{equation} \label{eq2}
\sum_{n \geq 0} \frac{|\{A,B \in Nil_n(F_q):AB=BA=0\}|}{|GL(n,q)|} u^n
= \frac{1}{(u/q)_{\infty}} \sum_{c \geq 0} \frac{u^{2c}}{q^{c^2} (1/q)_c (u/q)_c}.
\end{equation}

Here $Mat_n(F_q)$ is the set of $n \times n$ matrices over the finite field $F_q$, and $Nil_n(F_q)$ is the set of $n \times n$ nilpotent matrices over the finite field $F_q$. Also, we have used the (standard) notation
\[ (x)_i = (1-x)(1-x/q)(1-x/q^2) \cdots (1-x/q^{i-1}). \]

Our main purpose here is to show that these two generating functions can be rederived using statistical properties of a Cohen-Lenstra measure on partitions. Huang's lovely paper used analytic ideas related to Cohen-Lenstra heuristics, but our approach is probabilistic and different.

\section{Cohen-Lenstra random partitions}

To begin we give some notation. We let $\lambda$ be a partition of some non-negative integer $|\lambda|$ into parts $\lambda_1 \geq \lambda_2 \geq \cdots$. We let $m_i(\lambda)$ denote the number of parts of size $i$, and we define
\[ \lambda_i' = m_i(\lambda) + m_{i+1}(\lambda) + \cdots. \] So $\lambda_1'$ is the number of parts of $\lambda$, and for convenience we also denote this by $l(\lambda)$. Moreover, if one represents $\lambda$ by a diagram with row lengths $\lambda_1, \lambda_2, \cdots$, then $\lambda_i'$ is the size of the ith column of the diagram of $\lambda$.

In a very influential paper \cite{CL}, the number theorists Cohen and Lenstra defined a probability measure $P$ on the set of all partitions $\lambda$ of all natural numbers. The definition of the measure $P$ is given by the formula
\[ P(\lambda) = (1/q)_{\infty} \frac{1}{|Aut(\lambda)|}, \] where $Aut(\lambda)$ is the automorphism group of a finite abelian group of type $\lambda$. Although we won't use it, we mention the explicit formula
\[ |Aut(\lambda)| = q^{\sum_i (\lambda_i')^2} \prod_i (1/q)_{m_i(\lambda)} .\]

In later work, Fulman \cite{F1} studied a more general probability measure
\[ P_u(\lambda) = (u/q)_{\infty} \frac{u^{|\lambda|}}{|Aut(\lambda)|}. \]

We will use the following result from \cite{F2}, which gives a way to generate random partitions from the measure $P_u$.

\begin{theorem} \label{gen} Starting with $\lambda_0'=\infty$, define in succession $\lambda_1' \geq \lambda_2' \geq \cdots$ according to the rule that if $\lambda_i'=a$, then $\lambda_{i+1}'=b$ with probability
\[ K(a,b) = \frac{u^b (1/q)_a (u/q)_a}{q^{b^2} (1/q)_{a-b} (1/q)_b (u/q)_b}.\]
\end{theorem}

This gives the following corollary, the first part of which will be used in proving equation \eqref{eq1} and the second part of which will be used in proving equation \eqref{eq2}.

\begin{cor} \label{twosums}
\begin{enumerate}
\item For a non-negative integer $a$,
\[ \sum_{\lambda: \lambda_1'=a} P_u(\lambda) = \frac{(u/q)_{\infty}}{(u/q)_a} \frac{u^a}{q^{a^2} (1/q)_a} .\]
\item For non-negative integers $a \geq b$,
\[ \sum_{\lambda: \lambda_1'=a, m_1(\lambda)=b} P_u(\lambda) =
\frac{u^{2a-b} (u/q)_{\infty}}{q^{a^2+(a-b)^2} (1/q)_b (1/q)_{a-b} (u/q)_{a-b}}.\] \end{enumerate}
\end{cor}

\begin{proof} For the first result, Theorem \ref{gen} implies that
\[ \sum_{\lambda: \lambda_1'=a} P_u(\lambda) = K(\infty,a) .\]

For the second result, note that $m_1(\lambda) = \lambda_1' - \lambda_2'$. So Theorem \ref{gen} implies that
\[ \sum_{\lambda: \lambda_1'=a, m_1(\lambda)=b} P_u(\lambda) = K(\infty,a) K(a,a-b) \] and the result follows.
\end{proof}

\section{Mutually annihilating matrices}

The purpose of this section is to rederive \eqref{eq1} and \eqref{eq2}.

To begin, we recall the Jordan form of an element of $Mat_n(F_q)$. This associates to each monic, irreducible polynomial $\phi$ over $F_q$ a partition $\lambda_{\phi}$ such that \[ \sum_{\phi} d(\phi)|\lambda_{\phi}| = n, \] where $d(\phi)$ is the degree of $\phi$. For further background on Jordan forms over finite fields, one can consult Chapter 6 of \cite{He}.

Lemma \ref{counter} is proved in Stong \cite{St} and calculates the number of elements of $Mat_n(F_q)$ with given Jordan form.

\begin{lemma} \label{counter} Suppose that \[ \sum_{\phi} d(\phi) |\lambda_{\phi}| = n, \] so that $\{ \lambda_{\phi} \}$ is a possible Jordan form of an element of $Mat_n(F_q)$. Then the number of elements of $Mat_n(F_q)$ with Jordan form $\{ \lambda_{\phi} \}$ is equal to
\[ \frac{|GL(n,q)|}{\prod_\phi |Aut(\lambda_{\phi})|_{q \rightarrow q^{d(\phi)}}}.\]
\end{lemma}

Here the notation $|Aut(\lambda_{\phi})|_{q \rightarrow q^{d(\phi)}}$ means that we place $q$ by $q^{d(\phi)}$ in the formula for $|Aut(\lambda_{\phi})|$.

\subsection{Proof of \eqref{eq1}}

The following lemma is crucial to deriving \eqref{eq1}.

\begin{lemma} \label{howman} Let $A$ be an element of $Mat_n(F_q)$. Then the number of $B$ such that $AB=BA=0$ is equal to $q^{m^2}$, where $m$ is the number of Jordan blocks of $A$ with eigenvalue $0$.
\end{lemma}

\begin{proof} Clearly the number we are computing is a similarity invariant: if $A$ is replaced by $UAU^{-1}$, the set of $B$ such that $AB=BA=0$ is replaced by $UBU^{-1}$.
Now write $A=diag(C,N)$ where $C$ is invertible and $N$ is nilpotent (in Jordan form). If $AB=BA=0$, then $B$ commutes with $A$ and so $B = diag(0, L)$ with $NL=LN =0$.

So it reduces to looking at the nilpotent part of $A$. So assume that $A$ is nilpotent with Jordan blocks corresponding to a given partition and consider its centralizer.

If $B$ is in the centralizer, then we can write $B = (B_{ij})$ blocking it up with respect to the Jordan blocks. Then $AB=0$ if and only if the image of $B$ is in the kernel of $A$.
  Note that if we write $A = diag(J_1, ..., J_m)$ with the $J_i$ Jordan blocks then $AB = (J_iB_{ij})$ and there is a one dimensional space of possible $B_{ij}$ with this property (with the $B$ in the centralizer). Thus the set of $B$ with $AB=BA=0$ has dimension $m^2$ where m is the number of Jordan blocks.
\end{proof}

Now we proceed to the main result of this subsection.

\begin{proof} (Of \eqref{eq1}) It follows from Lemmas \ref{counter} and \ref{howman} that the number of pairs $A,B$ in $Mat_n(F_q)$ such that $AB=BA=0$ is equal to $|GL(n,q)|$ multiplied by
\[ \sum_{\{\lambda_{\phi}\}} \frac{q^{l(\lambda_z)^2} }{\prod_\phi |Aut(\lambda_{\phi})|_{q \rightarrow q^{d(\phi)}}}, \] where the sum is over all Jordan forms of $Mat_n(F_q)$ and $l(\lambda_z)$ is the number of parts of $\lambda_z$.

Separating out the contribution from the polynomial $\phi(z)=z$, we obtain $|GL(n,q)|$ multiplied by the coefficient of $u^n$ in
\[ \sum_{\lambda} q^{(\lambda_1')^2} \frac{u^{|\lambda|}}{|Aut(\lambda)|}
\cdot \prod_{\phi \neq z} \sum_{\lambda} \frac{u^{d(\phi)|\lambda|}}
{|Aut(\lambda)|_{q \rightarrow q^{d(\phi)}}} .\]

Now for any finite group, the sum over conjugacy classes of the reciprocal of centralizers sizes is equal to $1$. Applying this to $GL(n,q)$ gives that
\[ \prod_{\phi \neq z} \sum_{\lambda} \frac{u^{d(\phi)|\lambda|}}
{|Aut(\lambda)|_{q \rightarrow q^{d(\phi)}}} = \frac{1}{1-u} .\]

We conclude that
\[ \sum_{n \geq 0} \frac{|\{A,B \in Mat_n(F_q):AB=BA=0\}|}{|GL(n,q)|} u^n \]
is equal to \[ \frac{1}{1-u} \sum_{\lambda} q^{(\lambda_1')^2} \frac{u^{|\lambda|}}{|Aut(\lambda)|}.\] This is exactly
\[ \frac{1}{1-u} \frac{1}{(u/q)_{\infty}} \sum_{a \geq 0} q^{a^2}
\sum_{\lambda: \lambda_1'=a} P_u(\lambda).\] By part 1 of Corollary \ref{twosums},
this is equal to \[ \frac{1}{1-u} \sum_{a \geq 0} \frac{u^a}{(1/q)_a (u/q)_a}, \]
as claimed.
\end{proof}

\subsection{Proof of \eqref{eq2}}

The following lemma is crucial to deriving \eqref{eq2}.

\begin{lemma} \label{howman2} Let $A$ be a nilpotent element of $Mat_n(F_q)$. Then the number of nilpotent $B$ such that $AB=BA=0$ is equal to $q^{m^2-d}$, where $m$ is the number of Jordan blocks of $A$ and $d$ is the number of Jordan blocks of $A$ of size $1$.
\end{lemma}

\begin{proof}  We argue as in the proof of Lemma \ref{howman}.  So let $B=(B_{ij})$ commute with and annihilate $A$.  We saw there is a one dimensional
choice for each $B_{ij}$ independently.   It is well known that the centralizer of $A$ modulo the Jacobson radical is a direct product
$M_{d_1}(q) \times \ldots \times M_{d_r}(q)$ where $d_i$ is the number of Jordan blocks of size $i$.  Thus, our computation reduces to the case where
all Jordan blocks have the same size $d$.  If $i  >1$, we see that $BA=0$ implies that the image of $B$ is contained in the kernel of $A$ and so
$B$ is in the Jacobson radical.   In this case any $B$ with $AB=BA=0$ is nilpotent (and indeed $B^2=0$).   If $i=1$ and $d=d_1$,  then the centralizer is the full
matrix ring  $M_d(q)$ and it is well known (see \cite{FH} or \cite{Ge}) that the number of nilpotent matrices of size $d$ is $q^{d^2-d}$. This completes the proof.
\end{proof}

Now we proceed to the main result of this subsection.

\begin{proof} (Of \eqref{eq2}) From Lemma \ref{counter}, the number $n \times n$ of nilpotent matrices over $F_q$ with Jordan type $\lambda$ is equal to \[ \frac{|GL(n,q)|}{|Aut(\lambda)|}.\] So Lemma \ref{howman2} implies that
\[ \sum_{n \geq 0} \frac{|\{A,B \in Nil_n(F_q):AB=BA=0\}|}{|GL(n,q)|} u^n \] is equal to \[ \sum_{\lambda} \frac{u^{|\lambda|}}{|Aut(\lambda)|} q^{(\lambda_1')^2-m_1(\lambda)}, \] where the sum is over all partitions of all natural numbers.

This is equal to
\[ \frac{1}{(u/q)_{\infty}} \sum_{\lambda} P_u(\lambda) q^{(\lambda_1')^2-m_1(\lambda)}.\]

By part 2 of Corollary \ref{twosums}, this is equal to
\[ \sum_{a \geq b \geq 0} \frac{u^{2a-b}}{q^{(a-b)^2+b} (1/q)_b (1/q)_{a-b} (u/q)_{a-b}}.\] Letting $c=a-b$, this becomes
\begin{eqnarray*}
\sum_{b,c \geq 0} \frac{u^{2c+b}}{q^{c^2+b} (1/q)_b (1/q)_c (u/q)_c}
& = & \sum_{b \geq 0} \frac{u^b}{q^b (1/q)_b} \sum_{c \geq 0} \frac{u^{2c}}{q^{c^2} (1/q)_c (u/q)_c} \\
& = & \frac{1}{(u/q)_{\infty}} \sum_{c \geq 0} \frac{u^{2c}}{q^{c^2} (1/q)_c (u/q)_c}, \end{eqnarray*} where the last step used the well known identity
\[ \sum_{b \geq 0} \frac{u^b}{q^b (1/q)_b} = \frac{1}{(u/q)_{\infty}}.\]
\end{proof}

\section{Acknowledgments} Fulman was supported by Simons Foundation Grant 400528. Guralnick was supported by NSF Grant DMS-1901595 and Simons Fellowship 609771. We thank Yifeng Huang for useful discussions.

\end{document}